\documentclass[11pt]{amsart}
\usepackage{amscd,amssymb,amsmath,latexsym,graphicx,amsxtra}
\newtheorem{theorem}{Theorem}[section]
\newtheorem{lemma}[theorem]{Lemma}
\newtheorem*{thmM}{Theorem 3.2}
\newtheorem*{proM}{Proposition 2.1}
\def\S{{\mathbb S}}
\def\Z{{\mathbb Z}}
\theoremstyle{definition}

\newtheorem{example}[theorem]{Example}
\newtheorem{proposition}[theorem]{Proposition}
\newtheorem{corollary}[theorem]{Corollary}
\newtheorem{remark}[theorem]{Remark}
\numberwithin{equation}{section}
\def\T{{\mathbb T}}
\def\S{{\mathbb S}}

\begin{document}
\title{On the group structure of $[\Omega \mathbb S^2, \Omega Y]$}
\author{Marek Golasi\'nski}
\address{Institute of Mathematics, Casimir the Great University,
pl.\ Weyssenhoffa 11, 85-072 Bydgoszcz, Poland}
\email{marek@ukw.edu.pl}
\author{Daciberg Gon\c calves}
\address{Dept. de Matem\'atica - IME - USP, Caixa Postal 66.281 - CEP 05314-970,
S\~ao Paulo - SP, Brasil}
\email{dlgoncal@ime.usp.br}
\author{Peter Wong}
\address{Department of Mathematics, Bates College, Lewiston, ME 04240, U.S.A.}
\email{pwong@bates.edu}
\thanks{}

\begin{abstract} Let $J(X)$ denote the James
construction on a space $X$ and $J_n(X)$ be the $n$-th stage of the James filtration of $J(X)$. It is known that $[J(X),\Omega Y]\cong \lim\limits_{\leftarrow}
[J_n(X),\Omega Y]$ for any space $Y$. When $X=\mathbb S^1$, the circle, $J(\mathbb S^1)=\Omega \Sigma \mathbb S^1=\Omega \mathbb S^2$. Furthermore,
there is a bijection between $[J(\mathbb S^1),\Omega Y]$ and the product $\prod_{i=2}^\infty \pi_i(Y)$, as sets.
In this paper, we describe the group structure of $[J_n(\mathbb S^1),\Omega Y]$ by
determining the co-multiplication structure on the suspension $\Sigma J_n(\mathbb S^1)$.
\end{abstract}
\date {\today}
\maketitle

\section*{Introduction}\setcounter{section}{0}
 Groups of homotopy classes of maps $[\Omega \Sigma X, \Omega Y]$ from the loop space of the suspension $\Sigma X$ to the loop space of $Y$ play an important role in classical homotopy theory. These groups have been used to give functorial homotopy decompositions of loop suspensions via modular representation theory and to investigate the intricate relationship between Hopf invariants and looped Whitehead products. In the special case when $X=\S^1$ and $Y=\S^2$, the authors in \cite{cohen,cohen-sato,cohen-wu} explored the relationship between $[\Omega \S^2, \Omega \S^2]$ and the Artin's pure braid groups via Milnor's
free group construction $F[K]$ (see \cite{M2}) for a simplicial set $K$. For $K=\S^1$, the geometric realization of $F[\S^1]$ has the homotopy type of $\Omega \Sigma \S^1=J(\S^1)$, the James construction on $\S^1$. In general, the James construction $J(X)$ on $X$ admits a filtration $J_1(X) \subseteq J_2(X) \subseteq\cdots$ so that
$\displaystyle{[J(X),\Omega Y]\cong\lim_{\leftarrow} [J_n(X),\Omega Y]}$. The {\it Cohen groups} $[J_n(X),\Omega Y]$ and the {\it total Cohen group} $[J(X),\Omega Y]$ have been studied in \cite{wu2}
via the simplicial group $\{[X^n, \Omega Y]\}_{n\ge 1}$. In his original paper \cite{J}, James introduced $J(X)$ as a model for $\Omega\Sigma X$, the loop 
space of the suspension $\Sigma X$ of the space $X$ and showed that $\Sigma J(X)$ has the homotopy type of the suspension of the wedge of self smash products of $X$. This shows that $[\Omega \S^2, \Omega Y]=[J(\S^1),\Omega Y]$, {\it as a set}, is in one-to-one correspondence with the direct product $\prod_{i= 2}^\infty \pi_i(Y)$ of the higher homotopy groups of $Y$. However, the group structure of $[J(\S^1),\Omega Y]$ is far from being abelian.

In \cite{cohen-sato}, the group $[J_n(\S^1),\Omega Y]$ is shown to be a central extension with kernel $\pi_{n+1}(Y)$ and quotient
$[J_{n-1}(\S^1),\Omega Y]$. In \cite{ggw5}, it is shown that $\pi_{n+1}(Y)$ is in fact central in a larger group $\tau_{n+1}(Y)$,
the Fox torus homotopy group. The proof in \cite{ggw5} relies on embedding $[J_n(\S^1),\Omega Y]$ into $\tau_{n+1}(Y)$. Fox \cite{fox}
introduced the torus homotopy groups such that the Whitehead products, when embedded as elements of a torus homotopy group, become commutators.
Indeed, the Fox torus homotopy group $\tau_n(Y)$ is completely determined by the homotopy groups $\pi_i(Y)$ for $1\le i\le n$ and the Whitehead products.
Furthermore, Fox determined whether $\alpha \in \pi_{k+1}(Y), \beta\in \pi_{l+1}(Y)$, when embedded in
$\tau_n(Y)$ commute. Following \cite{fox}, we consider a $k$-subset ${\bf a}$ and an $l$-subset ${\bf b}$ of the set of indices $\{1,2,\ldots, n\}$
for some $n\ge k+l$. The sets ${\bf a}$ and ${\bf b}$ determine two embeddings $\pi_{i+1}(Y) \to \tau_{n+1}(Y)$ for $i=k,l$. Denote by $\alpha^{{\bf a}}$ and $\beta^{{\bf b}}$ the corresponding images of $\alpha$ and $\beta$ in $\tau_{n+1}(Y)$.

\begin{proM}
\begin{enumerate}
\item[(1)] If ${\bf a} \cap {\bf b}=\emptyset$ then $(\alpha^{{\bf a}},\beta^{{\bf b}})=(-1)^{w+(|{\bf a}|-1)}[\alpha, \beta]^{{\bf a}\cup {\bf b}}$.\\
\item[(2)] If ${\bf a} \cap {\bf b}\ne \emptyset$ then $(\alpha^{{\bf a}},\beta^{{\bf b}})=1$.
\end{enumerate}
Here $(x,y)$ denotes the commutator $xyx^{-1}y^{-1}$ of $x$ and $y$ and $w=\Sigma_{i\in {\bf a}, j\in {\bf b}} w_{i,j}$, where $w_{i,j}=1$ if $j<i$ and $w_{i,j}=0$ otherwise.
\end{proM}

Since the groups $[J_n(\S^1),\Omega Y]$ can be embedded as subgroups of the Fox torus homotopy groups $\tau_{n+1}(Y)$, the group structure of
$[J_n(\S^1),\Omega Y]$ is induced by that of $\tau_{n+1}(Y)$. In this paper, we obtain the group structure of $[J_n(\mathbb S^1),\Omega Y]$ via the torus homotopy group studied in \cite{ggw4} together with a recent result by Arkowitz and Lee \cite{AL2} on the co-$H$-structures of a wedge of spheres. The following is our main theorem.

\begin{thmM} The suspension co-$H$-structure $$\overline{\mu}_n : \Sigma J_n(\mathbb S^1)
\to \Sigma J_n(\mathbb S^1) \vee \Sigma J_n(\mathbb S^1)$$ for $n\ge 1$ is given by $\overline{\mu}_n k_i\simeq\iota_1k_i+\iota_2k_i+P_i$,
where the perturbation $P_i\simeq\sum_{l=0}^i\phi(i-l,i-1)P_{l,i}$  with the Fox function $\phi$ and
$$P_{l,i}: \mathbb S^{i+1}\to \mathbb S^{i-l+1}\vee \mathbb S^{l+1}\stackrel{k_{i-l}\vee k_l}{\hookrightarrow}\Sigma J_n(\mathbb S^1)\vee \Sigma J_n(\mathbb S^1)$$
determined by the Whitehead product map $\mathbb S^{i+1}\to \mathbb S^{i-l+1}\vee \mathbb S^{l+1}$ for $i=0,\ldots,n$ and $l=0,\ldots,i$.
\end{thmM}

Here the perturbations $P_i$ are used in \cite{AL2} to measure the deviation of the co-$H$-structure of a space of homotopy type of a wedge of spheres from the usual coproduct of the suspension co-$H$-structures of spheres. The function $\phi$ derived from Proposition \ref{gen-Fox-W-product} is used to determine the coefficients of Whitehead products.

This paper is organized as follows. In Section 1, we recall the Fox torus homotopy groups $\tau_{n+1}(Y)$ and give examples of $[J_n(\S^1),\Omega Y]$,
where the group structure can be obtained via $\tau_{n+1}(Y)$. In Section $2$, we obtain a closed form solution to a recurrence relation on the function
$\phi$ which gives the coefficient of the Whitehead product in Proposition \ref{gen-Fox-W-product}. Section $3$ combines the result of \cite{AL2}
and the result of Section $2$ to obtain the co-$H$-structure of $\Sigma J_n(\S^1)$. In Section 4, we revisit and generalize the examples of Section $1$.
In particular, we give necessary and sufficient conditions (see Proposition \ref{3_stem}) for $[J_{4n+1}(\S^1),\Omega \S^{2n}]$ to be abelian when $n$
is not a power of $2$. We also investigate when two torsion elements in $[J(\S^1),\Omega Y]$ commute. To end
this introduction, we want to point out that the Fox torus homotopy groups $\tau_{n+1}({\rm Conf}(n))$ of the
configuration space of $n$ distinct points in $\mathbb R^3$ were used by F.\ Cohen et al.\ \cite{cohen-et-al} to
give an alternative proof of a result of U.\ Koschorke, which is related  to Milnor's link homotopy \cite{M} and
homotopy string links examined by Habegger-Lin  \cite{H-L}, that the $\kappa$-invariant for the Brunnian links
in $\mathbb R^3$ is injective.

We thank the anonymous referee for his/her careful reading of the earlier version of the manuscript as well as helpful comments that lead to a better exposition of the paper.

\section{Fox torus homotopy groups}
In this section, we make some calculation on $[J_n(\S^1),\Omega Y]$ using the Fox torus homotopy groups. First, we recall from \cite{fox} the definition of the $n$-th Fox torus homotopy group of
a connected pointed space $Y$, for $n\ge 1$. Let $y_0$ be a basepoint of $Y$, then
$$
\tau_n(Y)\cong \tau_n(Y,y_0)=\pi_1(Y^{\T^{n-1}},\overline{y_0}),
$$
where $Y^{\T^{n-1}}$ denotes the space of unbased maps from the $(n-1)$-torus $\T^{n-1}$ to $Y$ and
$\overline{y_0}$ is the constant map at $y_0$. When $n=1$, $\tau_1(Y)=\pi_1(Y)$.

To re-interpret Fox's result, we showed in \cite{ggw1} that
$$
\tau_n(Y)\cong [F_n(\S^1),Y]
$$
the group of homotopy classes of basepoint preserving maps from the
reduced suspension $F_n(\S^1):=\Sigma (\T^{n-1}\sqcup *)$ of $\T^{n-1}$ adjoined with a distinguished point
to $Y$.

\addtocounter{theorem}{1}

One of the main results of \cite{fox} is the following split exact sequence:

\begin{equation}\label{fox-split}
0\to \prod_{i=2}^n \pi_i(Y)^{\sigma_i} \to \tau_n(Y) \stackrel{\dashleftarrow}{\to} \tau_{n-1}(Y) \to 0,
\end{equation}
where $\sigma_i=\binom{n-2}{i-2}$, the binomial coefficient.

With the isomorphism $\tau_{n-1}(\Omega Y)\cong \prod_{i=2}^n \pi_i(Y)^{\sigma_i}$ shown in \cite[Theorem 1.1]{ggw1},
the sequence \eqref{fox-split} becomes

\begin{equation}\label{general-fox-split}
0\to \tau_{n-1}(\Omega Y) \to \tau_n(Y) \stackrel{\dashleftarrow}{\to} \tau_{n-1}(Y) \to 0.
\end{equation}
 Here, the projection $\tau_n(Y) \to \tau_{n-1}(Y)$ is induced by the suspension of the inclusion
$\T^{n-2}\sqcup * \hookrightarrow \T^{n-1}\sqcup *$ given
by $(t_1,\ldots,t_{n-2})\mapsto (1,t_1,\ldots,t_{n-2})$ and the
section $\tau_{n-1}(Y) \to \tau_n(Y)$ is the homomorphism induced by the suspension of the
projection $\T^{n-1} \sqcup * \to \T^{n-2}\sqcup *$ given by $(t_1,\ldots,t_{n-1}) \mapsto (t_2,\ldots,t_{n-1})$. This
splitting (section) gives the semi-direct product structure so that the action
$$\bullet : \tau_{n-1}(Y)\times \tau_{n-1}(\Omega Y)\longrightarrow\tau_{n-1}(\Omega Y)$$
of the quotient $\tau_{n-1}(Y)$ on the kernel $\tau_{n-1}(\Omega Y)$ is simply conjugation
in $\tau_n(Y)$ by the image of $\tau_{n-1}(Y)$ under the section. It follows from the work
of Fox (in particular Proposition \ref{gen-Fox-W-product}) that the action is determined by the
Whitehead products. More precisely, given $\alpha \in \pi_i(Y), \beta \in \pi_j(Y)$, let
$\hat \alpha$ and $\hat \beta$ be the respective images of $\alpha$ in $\tau_{n-1}(Y)$ and
of $\beta$ in $\tau_{n-1}(\Omega Y)$. Then
$$\hat \alpha \bullet \hat \beta =\widehat {[\alpha,\beta]} \hat \beta,$$
where $\widehat{[\alpha,\beta]}$ denotes the image in $\tau_n(Y)$ of the Whitehead product $[\alpha,\beta]$.

Note that $F_n(\S^1)\simeq (\Sigma \T^{n-1}) \vee \S^1$ but
$[F_n(\S^1),Y]\cong [\Sigma \T^{n-1},Y] \rtimes \pi_1(Y)$.

\begin{example}\label{ex1}
The group $[J_2(\mathbb S^1),\Omega \mathbb S^2]$ is abelian (in fact isomorphic to $\mathbb Z \oplus \mathbb Z$, see Example \ref{ex1-revisit}). We now examine the multiplication in $[J_2(\mathbb S^1),\Omega \mathbb S^2]$ by embedding this inside the group $\tau_3(\mathbb S^2)$. First, we note that $\tau_3(\mathbb S^2) \cong (\pi_3(\mathbb S^2) \oplus \pi_2(\mathbb S^2)_{\{1\}}) \rtimes \pi_2(\mathbb S^2)_{\{2\}}$. Here we follow the work of \cite{fox} by using the indexing set $\{1,2\}$ for the two copies of $\pi_2(\mathbb S^2)$ in $\tau_3(\mathbb S^2)$. Let $\alpha, \beta \in \pi_2(\mathbb S^2)$ and we write $\alpha^{\{i\}}, \beta^{\{i\}} \in \pi_2(\mathbb S^2)_{\{i\}}$ for $i=1,2$. Using the semi-direct product structure and the representation of elements of $[J_2(\mathbb S^1),\Omega \mathbb S^2]$ in $\tau_3(\mathbb S^2)$, the product of two elements of $[J_2(\mathbb S^1),\Omega \mathbb S^2]$ is given by
\begin{equation*}
\begin{aligned}
&(\alpha^{\{1\}},\alpha^{\{2\}})\cdot (\beta^{\{1\}}, \beta^{\{2\}})\\
=&(\alpha^{\{1\}}+(\alpha^{\{2\}}\bullet \beta^{\{1\}}), \alpha^{\{2\}}+\beta^{\{2\}})\\
=&(\alpha^{\{1\}}+\beta^{\{1\}}+(-1)^{w+(m-1)}[\alpha, \beta]^{\{1,2\}}, \alpha^{\{2\}}+\beta^{\{2\}}) \quad \text{(here $m=1,w=1$)} \\
=&(\alpha^{\{1\}}+\beta^{\{1\}}+(-1)[\alpha, \beta]^{\{1,2\}}, \alpha^{\{2\}}+\beta^{\{2\}}).
\end{aligned}
\end{equation*}
Conversely,
\begin{equation*}
\begin{aligned}
&(\beta^{\{1\}}, \beta^{\{2\}})\cdot (\alpha^{\{1\}},\alpha^{\{2\}})\\
=&(\beta^{\{1\}}+\alpha^{\{1\}}+(-1)^{w+(m-1)}[\beta,\alpha]^{\{1,2\}}, \beta^{\{2\}}+\alpha^{\{2\}}) \quad \text{here $m=1$} \\
=&(\alpha^{\{1\}}+\beta^{\{1\}}+(-1)[\beta,\alpha]^{\{1,2\}}, \alpha^{\{2\}}+\beta^{\{2\}}).
\end{aligned}
\end{equation*}
Since both $\alpha$ and $\beta$ have dimension $2$, the Whitehead product $[\alpha,\beta]$ coincides with $[\beta, \alpha]$.
Thus the above calculation shows that $[J_2(\mathbb S^1),\Omega \mathbb S^2]$ is abelian.
\end{example}

However, the groups $[J_n(X),\Omega Y]$ are non-abelian in general.
\begin{example}\label{ex2}
The group $[J_3(\mathbb S^1), \Omega Y]$ is non-abelian for $Y=\mathbb S^2 \vee \mathbb S^3$. First, the Fox homotopy group $\tau_4(Y)\cong \tau_3(\Omega Y) \rtimes \tau_3(Y)$. Since $Y$ is 1-connected, using the Fox short split exact sequence \eqref{general-fox-split},
we obtain
\begin{equation}
\begin{aligned}
\tau_4(Y) &\cong \left( \pi_4(Y) \oplus 2\pi_3(Y) \oplus \pi_2(Y)\right)\rtimes \left((\pi_3(Y)\oplus \pi_2(Y))\rtimes \pi_2(Y)\right).
\end{aligned}
\end{equation}
Let $\iota_1:\mathbb S^2 \hookrightarrow \mathbb S^2 \vee \mathbb S^3$ and $\iota_2:\mathbb S^3 \hookrightarrow \mathbb S^2 \vee \mathbb S^3$ denote the canonical inclusions. The (basic) Whitehead product $[\iota_1,\iota_2]$ is a non-trivial element in $\pi_4(\mathbb S^2 \vee \mathbb S^3)$. When regarded as an element of the group $\tau_4(Y)$, the image of $[\iota_1,\iota_2]$ in $\tau_4(Y)$ is the commutator $(\hat a, \hat b)$, where
\begin{equation}
\begin{aligned}
\hat a&=\left( (1\oplus (1\oplus 1) \oplus a), ((1\oplus a),a)\right) \\
\hat b&=\left( (1\oplus (b\oplus b) \oplus 1), ((b\oplus 1),1)\right)
\end{aligned}
\end{equation}
with $a=[\iota_1]\in \pi_2(Y)$ and $b=[\iota_2]\in \pi_3(Y)$. By \cite[Theorem 2.2]{ggw5}, $\hat a, \hat b$ are
in the image of $[J_3(\mathbb S^1),\Omega Y]\hookrightarrow\tau_4(Y)$. Since the commutator $(\hat a, \hat b)$
is non-trivial, it follows that $[J_3(\mathbb S^1),\Omega Y]$ is non-abelian.
\end{example}

Now, we analyze the central extension
\begin{equation}\label{cohen-sequence}
0 \to \pi_{n+1}(Y) \to [J_n(\mathbb S^1),\Omega Y] \to [J_{n-1}(\mathbb S^1),\Omega Y] \to 0
\end{equation}
and give an example in which this extension does not split.
The calculation below makes use of torus homotopy groups.

\begin{example}\label{ex3}
Take $Y=\mathbb S^4$. Since $\S^4$ is $3$-connected, we have
$[J_1(\mathbb S^1),\Omega Y]=[J_2(\mathbb S^1),\Omega Y]=0$ and
$[J_3(\mathbb S^1),\Omega \mathbb S^4]\cong \pi_4(\mathbb S^4)\cong \mathbb Z$.
For $n=4$, the sequence
\eqref{cohen-sequence} becomes
$$0\to \pi_5(\mathbb S^4) \to [J_4(\mathbb S^1),\Omega \mathbb S^4] \to [J_3(\mathbb S^1),\Omega \mathbb S^4] \to 0.
$$
Note that the corresponding Fox split exact sequence is
$$
0\to \pi_5(\mathbb S^4)\oplus 3 \pi_4(\mathbb S^4) \to \tau_5(\mathbb S^4) \stackrel{\dashleftarrow}{\to} \tau_4(\mathbb S^4) \to 0
$$
and $\tau_4(\mathbb S^4)\cong \pi_4(\mathbb S^4)$. For dimensional reasons, $\tau_5(\mathbb S^4)$ does not contain any non-trivial
Whitehead products so that $\pi_5(\mathbb S^4)\oplus 3 \pi_4(\mathbb S^4)$ is central in $\tau_5(\mathbb S^4)$.
Thus, $\tau_5(\mathbb S^4)\cong \pi_5(\mathbb S^4)\oplus 4 \pi_4(\mathbb S^4)$ is abelian. It follows
that $$[J_4(\mathbb S^1),\Omega \mathbb S^4]\cong \pi_5(\S^ 4)\oplus\pi_4(\S^4)$$ is abelian.

For $n=5$, the corresponding Fox split exact sequence is
$$
0\to \pi_6(\mathbb S^4)\oplus 4 \pi_5(\mathbb S^4) \oplus 6 \pi_4(\mathbb S^4) \to \tau_6(\mathbb S^4) \stackrel{\dashleftarrow}{\to} \tau_5(\mathbb S^4) \to 0.$$
Again, for dimensional reasons, $\tau_6(\mathbb S^4)$ contains no non-trivial Whitehead products so that
$$[J_5(\S^1),\Omega \S^4]\cong \pi_6(\S^4) \oplus \pi_5(\S^4) \oplus \pi_4(\S^4).$$ It follows that $[J_5(\mathbb S^1),\Omega\mathbb S^4]$ is
abelian and the sequence \eqref{cohen-sequence} splits for $n=5$. To see this, we note that $[J_5(\mathbb S^1),\Omega\mathbb S^4]$ is of rank $1$ and
so it is either $\mathbb Z \oplus \mathbb Z_2 \oplus \mathbb Z_2$ or $\mathbb Z \oplus \mathbb Z_4$, where $\mathbb{Z}_n$ is the
cyclic group of order $n$. Since $\tau_6(\mathbb S^4)$ has no elements of order $4$ so we have $[J_5(\mathbb S^1),\Omega\mathbb S^4] \cong \mathbb Z \oplus \mathbb Z_2 \oplus \mathbb Z_2$.

When $n=6$, the sequence \eqref{cohen-sequence} becomes
$$
0\to \pi_7(\mathbb S^4) \to [J_6(\mathbb S^1),\Omega \mathbb S^4] \to [J_5(\mathbb S^1),\Omega \mathbb S^4] \to 0.
$$
By projecting $[J_6(\mathbb S^1),\Omega \mathbb S^4]$ onto $[J_3(\mathbb S^1),\Omega \mathbb S^4]$, the above sequence
gives rise to the following exact sequence
\begin{equation}\label{alt-exact}
0\to \mathcal W \to [J_6(\mathbb S^1),\Omega \mathbb S^4] \to [J_3(\mathbb S^1),\Omega \mathbb S^4]\cong \pi_4(\mathbb S^4)\cong \mathbb Z \to 0
\end{equation}
so that $[J_6(\mathbb S^1),\Omega \mathbb S^4] \cong \mathcal W \rtimes \pi_4(\mathbb S^4)$. Similarly, the corresponding Fox
split exact sequence can be written as
$$
0\to \widehat {\mathcal W} \to \tau_7(\S^4) \to \tau_4(\mathbb S^4)=\pi_4(\mathbb S^4) \to 0.
$$
Here, $\mathcal W$ is generated by elements of $\pi_i(\mathbb S^4)$ for $i=5,6,7$ while $\widehat {\mathcal W}$ is generated by elements of $\pi_i(\mathbb S^4)$ for $i=4,5,6,7$. It follows
that the action of $\pi_4(\mathbb S^4)$ on $\mathcal W$ is the same as the action of $\pi_4(\mathbb S^4)=\tau_4(\mathbb S^4)$ on
$\widehat {\mathcal W}$ and is determined by the Whitehead products. For dimensional reasons, if $x\in \mathcal W$, the Whitehead
product of $x$ with any element in $\pi_4(\mathbb S^4)$ will be zero in $\tau_7(\mathbb S^4)$ and hence $\pi_4(\mathbb S^4)$
acts trivially on $\mathcal W$. Thus,
$$
[J_6(\mathbb S^1),\Omega \mathbb S^4] \cong \mathcal W \times \pi_4(\mathbb S^4)\cong \bigoplus_{i=4}^7 \pi_i(\mathbb S^4)
$$
while $\tau_7(\mathbb S^4)$ is non-abelian. Moreover, the central extension
$$
0\to \pi_7(\mathbb S^4) \to [J_6(\mathbb S^1),\Omega \mathbb S^4] \to [J_5(\mathbb S^1),\Omega \mathbb S^4] \to 0
$$
splits because $\mathcal W\subset \pi_5(\mathbb S^4) \oplus \pi_6(\mathbb S^4) \subset \pi_4(\mathbb S^4) \oplus \pi_5(\mathbb S^4) \oplus \pi_6(\mathbb S^4) \cong [J_5(\mathbb S^1),\Omega \mathbb S^4]$.

Next, consider the case when $n=7$. Then, the sequence \eqref{cohen-sequence} becomes
$$
0\to \pi_8(\mathbb S^4) \to [J_7(\mathbb S^1),\Omega \mathbb S^4] \to [J_6(\mathbb S^1),\Omega \mathbb S^4] \to 0.
$$
The corresponding Fox split exact sequence is
$$
0\to \pi_8(\S^4) \oplus 6 \pi_7(\S^4) \oplus 15 \pi_6(\S^4)\oplus
20 \pi_5(\S^4) \oplus 15 \pi_4(\S^4) \to \tau_8(\mathbb S^4) \stackrel{\dashleftarrow}{\to} \tau_7(\mathbb S^4) \to 0.
$$
Again, for dimensional reasons, the only non-trivial Whitehead products lie in $\pi_8(\S^4)$ between the elements in $\pi_4(\S^4)$ and
$\pi_5(\S^4)$. Let $\iota_1, \iota_2$ be the generators of the cyclic groups $\pi_4(\mathbb S^4)\cong \mathbb Z$ and
$\pi_5(\mathbb S^4)\cong \mathbb Z_2$, respectively. Since there are $15$ copies of $\pi_5(\mathbb S^4)$ and $20$ copies of $\pi_4(\S^4)$ in
$\tau_7(\S^4)$, there are a total of $35$ copies of $\pi_4(\S^4)$ and $35$ copies of $\pi_5(\S^4)$ in $\tau_8(\S^4)$. According to \cite{fox},
the Whitehead products are determined by embedding $\pi_n(\S^4)$ in $\tau_r(\S^4)$ using $\binom{r-1}{n-1}$ embeddings. With $r=8$ and
$n=4,5$, there are $(35)^2$ possible pairings $(\iota_1',\iota_2')$, where $\iota_i'$ corresponds to the image of $\iota_i$
under one of $35$ embeddings.
\par Now, by the result of \cite{fox}, once an embedding for $\pi_4(\S^4)$ is chosen, there is a unique
embedding of $\pi_5(\S^4)$ so that non-trivial Whitehead products can be formed. Thus, there are exactly $35$ such products (commutators)
each of which is the generator of $\pi_8(\S^4)\cong \mathbb Z_2$. The product of these $35$ commutators
is non-trivial since $35\not\equiv\, 0\,(\bmod\, 2)$. If $\tilde \iota_i$ is a preimage of $\iota_i$ in
$[J_7(\mathbb S^1), \Omega \mathbb S^4]$ for $i=1,2$, the commutator $[\tilde \iota_1, \tilde \iota_2]$
is independent of choice of the preimages since $\pi_8(\mathbb S^4)$ is central in $[J_7(\mathbb S^1), \Omega \mathbb S^4]$.
This commutator is non-trivial in $[J_7(\mathbb S^1), \Omega \mathbb S^4]$ and this shows that the projection
$[J_7(\mathbb S^1), \Omega \mathbb S^4] \to [J_6(\mathbb S^1), \Omega \mathbb S^4]$ cannot have a section.
\par We point out that such non-trivial Whitehead products will persist and induce non-trivial commutators in
$[J_k(\mathbb S^1), \Omega \mathbb S^4]$ for any $k>7$. Thus, we conclude that $[J_k(\mathbb S^1), \Omega \mathbb S^4]$ is non-abelian for any $k\ge7$.
\end{example}

\section{The Fox number and the function $\phi$}
The Fox torus homotopy group $\tau_n(Y)$ is completely determined by the homotopy groups $\pi_i(Y)$ for $1\le i\le n$ and the Whitehead products.
Furthermore, Fox determined whether $\alpha \in \pi_{k+1}(Y), \beta\in \pi_{l+1}(Y)$, when embedded in
$\tau_n(Y)$ commute. Following \cite{fox}, we consider a $k$-subset ${\bf a}$ and an $l$-subset ${\bf b}$ of the set of indices $\{1,2,\ldots, n\}$
for some $n\ge k+l$. The sets ${\bf a}$ and ${\bf b}$ determine two embeddings $\pi_{i+1}(Y) \to \tau_{n+1}(Y)$ for $i=k,l$. Denote by $\alpha^{{\bf a}}$ and $\beta^{{\bf b}}$ the corresponding images of $\alpha$ and $\beta$ in $\tau_{n+1}(Y)$.

\begin{proposition}\label{gen-Fox-W-product}{\em
\begin{enumerate}
\item[(1)] If ${\bf a} \cap {\bf b}=\emptyset$ then $(\alpha^{{\bf a}},\beta^{{\bf b}})=(-1)^{w+(|{\bf a}|-1)}[\alpha, \beta]^{{\bf a}\cup {\bf b}}$.\\
\item[(2)] If ${\bf a} \cap {\bf b}\ne \emptyset$ then $(\alpha^{{\bf a}},\beta^{{\bf b}})=1$.
\end{enumerate}
Here $(x,y)$ denotes the commutator $xyx^{-1}y^{-1}$ of $x$ and $y$ and $w=\Sigma_{i\in {\bf a}, j\in {\bf b}} w_{i,j}$, where $w_{i,j}=1$ if $j<i$ and $w_{i,j}=0$ otherwise.}
\end{proposition}
The number $(-1)^{w+(|{\bf a}|-1)}$ as in Proposition \ref{gen-Fox-W-product}, which is crucial in determining the structure of the torus homotopy groups, depends on the parameters $l$ and $k$ and we shall call this the {\it Fox number} of the partition $\{{\bf a},{\bf b}\}$.

For any integers $l,k$ with $k>l>0$, let
$$
\phi(l,k)=\sum_{{\bf a}, |{\bf a}|=l} (-1)^{w+(|{\bf a}|-1)}.
$$
Moreover, for all $k\ge 0$, we let $\phi(k,k)=(-1)^k$ and $\phi(0,k)=1$. We shall call $\phi$ the {\it Fox function}.

In order to compute $\phi(l,k)$ in terms of a simple algebraic expression, we state our main lemma.
\begin{lemma}\label{main-lemma} The function $\phi$ satisfies the following recurrence relations.
For $k>l>0$,
$$\phi(l,k)=(-1)^{k-l+1}\phi(l-1,k-1)+\phi(l,k-1)$$
and
$$\phi(k,k)=-\phi(k-1,k-1).$$
\end{lemma}
\begin{proof} The formula for $l=k$ follows from the definition of  $\phi(k , k)$.
Now assume that  $k>l>0$.  Let $L$ denote  the family of all subsets of   $\{ 1,2,\ldots,k\}$
with cardinality $l$. Divide this family into two subfamilies, $L_1$ and $L_2$. A subset
belongs to $L_1$ if it contains $k$, otherwise it belongs to $L_2$. Recall that the Fox number for one partition is given by $(-1)^{l-1+w}$. Summing of the Fox numbers over all  partitions $\{{\bf a},{\bf b}\}$ with ${\bf a}$ belonging  to $L_2$ amounts to
computing the Fox function for the pair  $(l,k-1)$  since the element
$k$ does not play a role in the calculation because $k$ is the last element of the set of $k$ elements and it does not belong to any subset of $l$ elements. Thus we conclude that the Fox function restricted to those partitions, where ${\bf a}$ belongs to $L_2$ coincides with $\phi(l,k-1)$.

For the  sum of the Fox numbers over all  partitions with ${\bf a}$ belonging to $L_1$, since $k$ belongs to the subset ${\bf a}$ the number $w$ contains a summand $k-l$ independent of the subset in $L_1$ since $\sum_j w_{k,j}=\sum_{j\in {\bf b}}1=k-l$. Then the remaining part of $w$ is obtained by considering subsets of $l-1$ elements in a set of cardinality $k-1$. So this coincides with
the calculation of $\phi(l-1,k-1)$ except that the computation for $\phi(l-1,k-1)$ uses subsets of length
$l-1$ and the one for elements of $L_1$ uses subsets of length $l$.  Thus we conclude that the Fox number  restricted
to those partitions with ${\bf a}$ belonging to $L_1$ coincides with $  (-1)^{k-l+1}\phi(l-1,k-1)$ and the result follows.
\end{proof}

Our function $\phi$ by  definition satisfies $\phi(0,k)=1$ and certainly satisfies the equality  $\phi(0,2k+1)=\phi(0,2k)$.  It is not difficult to see from
the definition of $\phi$ that $\phi(1,2k)=0$ and

$\phi(1,2k+1)=-\phi(0,2k)=-1$.

The following three propositions give the basic properties in order to compute  $\phi$.

\begin{proposition}\label{odd-even}  {\em \mbox{\em (1)} For  $l$ odd and $k>l/2$ we have  $$\phi(l,2k)=0.$$

\mbox{\em (2)} For  $l$ even  and  $k\geq l/2$ we have  $$\phi(l,2k)=\phi(l,2k+1).$$
\mbox{\em (3)} For  $l$ even  and $k\geq l/2$ we have $$ \phi(l+1,2k+1)=-\phi(l,2k+1).$$}
\end{proposition}
\begin{proof} The proof is by induction. We say that
the inductive hypothesis holds for an integer $m$ if (1) holds for all $l\leq m$ with $l$ odd, and (2) and (3) hold for all $l\leq m$ with $l$ even.

First, we show that the inductive hypothesis holds for  $m=1$. Part (2) follows from the definition of $\phi$ where both sides of the equation  are $1$.
For the parts ($1$) and ($3$) we use the following equations
\begin{equation}\label{formulaI}
\phi(1,2k+1)=(-1)^{2k+1}\phi(0,2k)+\phi(1,2k)
\end{equation}
and
\begin{equation}\label{formulaII}
\phi(1,2k)=(-1)^{2k}\phi(0,2k-1)+\phi(1,2k-1)
\end{equation}
obtained  from Lemma \ref{main-lemma}.
\par By induction on $k$, we prove simultaneously that   $\phi(1,2k)=0$,
and  $\phi(1,2k-1)=0-1=-1$, where the latter equality is equivalent to part (3) for $l=0$.
By definition  $\phi(1,1)=-1$ and by formula \eqref{formulaII} above
$\phi(1,2)=(-1)^{2}\phi(0,1)+\phi(1,1)=1-1=0$. So the result holds for $k=1$.
Suppose that  $\phi(1,2r)=0$, $\phi(1,2r-1)=-1$, for $r\leq k$.
Next we prove that $\phi(1,2k+2)=0$ and  $\phi(1,2k+1)=-1$.  By the inductive hypothesis and equation \eqref{formulaI} it follows that
$\phi(1,2k+1)=(-1)^{2k+1}\phi(0,2k)+\phi(1,2k)=-1+0=-1$.  Now the inductive hypothesis and formula \eqref{formulaII} yield
$\phi(1,2k+2)=(-1)^{2k+2}\phi(0,2k+1)+\phi(1,2k+1)=1-1=0$ and the result follows.

Now assume the assertions for parts (1) - (3) hold for $m=2s+1$. Then, we show that the result holds for $m=2s+3$.
The proof is similar to the arguments above.  First, we have
$$\phi(2s+2,2k+1)=\phi(2s+1,2k)+\phi(2s+2,2k)=\phi(2s+2,2k)$$
where the first equality follows from Lemma \ref{main-lemma} and the second equality holds by inductive hypothesis about part (1). So part (2) follows.
For parts ($1$) and ($3$), we use the following equations
\begin{equation}\label{formulaIII}
\phi(2s+3,2k+1)=-\phi(2s+2,2k)+\phi(2s+3,2k),
\end{equation}
\begin{equation}\label{formulaIV}
\phi(2s+3,2k)=\phi(2s+2,2k-1)+\phi(2s+3,2k-1)
\end{equation} and
\begin{equation}\label{formulaV}
\phi(2s+2,2k+1)=\phi(2s+1,2k)+\phi(2s+2,2k)
\end{equation}
obtained  from Lemma \ref{main-lemma}.
\par By induction on $k$, we  prove simultaneously that   $\phi(2s+3,2k)=0$,
and  $\phi(2s+3,2k-1)=-\phi(2s+2,2k-1)$. We have $k\geq s+2$ and take
$k=s+2$.
By definition of $\phi$, we have
$\phi(2s+3,2s+3)=-\phi(2s+2,2s+2)$ and by  equation \eqref{formulaV}, $\phi(2s+2,2s+3)=\phi(2s+1,2s+2)+\phi(2s+2,2s+2)=\phi(2s+2,2s+2)$ where the last equality
follows from the inductive hypothesis.  Therefore $\phi(2s+3,2s+3)=-\phi(2s+2,2s+3)$ and (3) follows. The following equation holds
$$\phi(2s+3,2s+4)=\phi(2s+2,2s+3)+\phi(2s+3,2s+3)=-\phi(2s+3,2s+3)+\phi(2s+3,2s+3)=0,$$
where the first equality follows from equation \eqref{formulaIV}, and the second equality
follows from  part (3). So the  result holds for part  (1).

Now, suppose that the statement holds for $k$ and
let us prove for $k+1$.
From equation \eqref{formulaIII} we have
$\phi(2s+3,2k+1)=-\phi(2s+2,2k)+\phi(2s+3,2k)$. Since $\phi(2s+3,2k)=0$ by inductive hypothesis,
and $\phi(2s+2,2k)=\phi(2s+2, 2k+1)$ it follows that (3) holds. It remains to show that $\phi(2s+3,2k+2)=0$ in order for (1) to hold.
From  \eqref{formulaIV} we have $\phi(2s+3,2k+2)=\phi(2s+2,2k+1)+\phi(2s+3,2k+1)$, and from (3), $\phi(2s+3,2k+1)-\phi(2s+2,2k+1)$, so it follows that
$\phi(2s+3,2k+2)=0$.
Therefore  (1) and (3) hold for $k+1$ and this concludes the proof.
  \end{proof}

  \begin{proposition}\label{rec} {\em The function $\phi$ satisfies the recursive formula
\begin{equation}\label{ggw-formula}
\phi(2l,2k)=\phi(2l ,2k-2)+\phi(2l -2, 2k-2).
\end{equation}}
\end{proposition}
\begin{proof}  From Lemma \ref{main-lemma}, we have
$$\phi(2l,2k)=(-1)^{2k-2l+1}\phi(2l-1,2k-1)+\phi(2l,2k-1)=(-1)\phi(2l-1,2k-1)+\phi(2l,2k-1).$$
From Proposition \ref{odd-even}, we have
$$(-1)\phi(2l-1,2k-1)=\phi(2l-2,2k-2)$$
and
$$\phi(2l,2k-1)=\phi(2l,2k-2).$$
Hence, the result follows.
\end{proof}

Next, we give a simple expression for $\phi(l,k)$ when $l$ and $k$ are even.
\begin{proposition}\label{Pascal}{\em
For any $l,k\ge 1$,
$$
\phi(2l,2k)=-\binom{k}{l}.
$$}
\end{proposition}
\begin{proof}
Let $c(n,k)$ be a function of two integer variables such that
\begin{equation}\label{2-var}
c(n+1,k)=c(n,k)+c(n,k-1)
\end{equation}
for $n,k\ge 1$. Suppose that $c(n,0)=a_n$ and $c(1,k)=b_k$, where $\{a_n\}_{n\ge 1}$ and $\{b_k\}_{k\ge 1}$ are two arbitrary sequences. Then H. Gupta \cite{gupta} showed that
\begin{equation}\label{gupta}
c(n,k)=\sum_{r=k}^{n-1} \binom{r-1}{k-1} a_{n-r} + \sum_{r=0}^{k-1} \binom{n-1}{r} b_{k-r}.
\end{equation}
Now, if we let $c(n,k)=\phi(2k,2n)$ then \eqref{ggw-formula} shows that \eqref{2-var} holds. Moreover, $a_n=c(n,0)=\phi(0,2n)=-1$ for all $n\ge 0$
and $b_k=c(1,k)=\phi(2k,2)$. It follows that $b_1=\phi(2,2)=-1$ and $b_k=0$ for all $k\ge 2$. Now, the formula \eqref{gupta} becomes
\begin{equation*}
\phi(2l , 2k)=\sum_{r=l}^{k-1} \binom{r-1}{l-1} a_{k-r} + \sum_{r=0}^{l-1} \binom{k-1}{r} b_{l-r}.
\end{equation*}
Since $b_k=0$ for $k\ge 2$, it follows that
\begin{equation}\label{ggw-recurrence}
\begin{aligned}
\phi(2l , 2k)&=\sum_{r=l}^{k-1} \binom{r-1}{l-1} a_{k-r} + \binom{k-1}{l -1} b_{1} \\
                &=\sum_{r=l}^{k-1} \binom{r-1}{l-1} (-1) + \binom{k-1}{l -1} (-1) \\
                &=-\sum_{r=l}^{k} \binom{r-1}{l-1}.
\end{aligned}
\end{equation}
The following equality
\begin{equation}\label{pascal}
\binom{0}{k} + \binom{1}{k} + \cdots + \binom{n}{k}=\binom{n+1}{k+1}
\end{equation}
can be derived using the Pascal triangle and the general form of the binomial coefficient $\binom{n}{k}$.
Now,
\begin{equation*}
\begin{aligned}
\sum_{r=l}^{k} \binom{r-1}{l-1}&=\binom{l -1}{l -1} + \cdots +\binom{k-1}{l -1}\\
                                     &=\left[\binom{0}{l -1} + \cdots +\binom{l -2}{l -1} + \binom{l -1}{l -1} + \cdots +\binom{k-1}{l -1}\right] - \left[\binom{0}{l -1} + \cdots +\binom{l -2}{l -1}\right] \\
                                     &=\binom{k}{l} -\binom{l -1}{l} \quad \text{by \eqref{pascal}} \\
                                     &=\binom{k}{l} -0 = \binom{k}{l}.
\end{aligned}
\end{equation*}
Hence, we have
\begin{equation}\label{simple-phi}
\phi(2l, 2k)=-\binom{k}{l}
\end{equation}
and the proof is complete.
\end{proof}

Lemma \ref{main-lemma} and all propositions in this section yield its main result.

\begin{theorem}\label{main-phi}
For any integers $l, k$ with $1\le l \le k$, we have
$$
\phi(l,k) \quad = \quad
\left\{
\aligned
& -\binom{\frac{k}{2}}{\frac{l}{2}}, \qquad & \text{if $l$ is even and $k$ is even;} \\
& 0, \qquad & \text{if $l$ is odd and $k$ is even;} \\
& \binom{\frac{k-1}{2}}{\frac{l-1}{2}}, \qquad & \text{if $l$ is odd and $k$ is odd;} \\
& -\binom{\frac{k-1}{2}}{\frac{l}{2}}, \qquad & \text{if $l$ is even and $k$ is odd.}
\endaligned
\right.
$$
\end{theorem}

\section{Group structure of $[J_n(\mathbb S^1),\Omega Y]$}
The group structure of $[J_n(\S^1),\Omega Y]$ is determined by the suspension co-$H$-structure on $\Sigma J_n(\S^1)$.
But, in view of \cite{ggw5}, the suspension $p_n : F_{n+1}(\S^1)\to \Sigma J_n(\S^1)$ of the projection map $\T^n\sqcup *\to J_n(\S^1)$
leads to a monomorphism of groups $$[\Sigma J_n(\S^1),Y] \hookrightarrow [F_{n+1}(\S^1),Y]=\tau_{n+1}(Y)$$ for any pointed space $Y$.
Thus, the group structure of $[\Sigma J_n(\S^1),Y]$ is detrmined by the multiplication of  $[F_{n+1}(\S^1),Y]$.
\par To relate those structures, first notice that the cofibration
$$\S^1\stackrel{j}{\to} \mathcal P_n(\S^1):=\T^n/\T^{n-1}\stackrel{q}{\to} \S^1\wedge \T^{n-1}$$
has a retraction $p : \mathcal P_n(\S^1)\to \S^1$.
Hence, the map $$\Sigma p+\Sigma q :\Sigma \mathcal P_n(\S^1)\stackrel{\simeq}{\longrightarrow} \Sigma \S^1\vee \S^1\wedge \Sigma\T^{n-1}$$
is a homotopy equivalence for $n\ge 1$. Because $\Sigma(X_1\times X_2)\simeq \Sigma X_1\vee \Sigma X_2\vee \Sigma(X_1\wedge X_2)$
for any pointed spaces $X_1$ and $X_2$, by an inductive argument, we derive:
\begin{equation}\label{eq1}
\begin{array}{l}
\Sigma \mathcal P_n(\S^1)\simeq \bigvee_{k=1}^n\binom{n-1}{k-1}\S^{k+1},\\
\vspace{3mm}
F_n(\S^1)\simeq \Sigma\T^{n-1}\vee \S^1 \simeq\bigvee_{k=0}^{n-1}\binom{n-1}{k}\S^{k+1}
\end{array}
\end{equation}
for $n\ge 1$.
Further, recall from \cite{J} that
\begin{equation}\label{eq2}
\Sigma J_n(\S^1)\simeq\bigvee_{k=1}^n\S^{k+1}
\end{equation}
for $n\ge 1$ and notice that (up to homotopy equivalences above) the suspension map $$p_n : F_{n+1}(\S^1)\longrightarrow \Sigma J_n(\S^1)$$
restricts to ${p_n}|_{\S^1}=\ast$ and $p_n|_{\S^{k+1}} : \S^{k+1} \to \Sigma J_n(\S^1)$ to the inclusion map for $k=1,\ldots,n$.
\par In view of \cite[Theorem 3.1]{ggw1}, the suspension co-$H$-structure $$\hat{\mu}_n : F_n(\S^1)\to F_n(\S^1)\vee F_n(\S^1)$$ on $F_n(\S^1)$ leads to
the following split exact sequence
$$
1\to [\Sigma {\mathcal P}_{n-1}(\S^1),Y] \to [F_n(\S^1),Y]\stackrel{\dashleftarrow}{\to} [F_{n-1}(\S^1),Y] \to 1
$$
for any pointed space $Y$.

Hence, $[F_n(\S^1),Y]\cong [\Sigma \mathcal P_{n-1}(\S^1),Y]\rtimes[F_{n-1}(\S^1),Y]$ is
the semi-direct product with respect to the natural action
$$[F_{n-1}(\S^1),Y]\times[\Sigma \mathcal P_{n-1}(\S^1),Y]\to[\Sigma \mathcal P_{n-1}(\S^1),Y].$$
In particular, for $Y=F_{n-1}(\S^1)\vee\Sigma \mathcal P_{n-1}(\S^1)$, by the natural bijection
$[F_{n-1}(\S^1),Y]\times[\Sigma \mathcal P_{n-1}(\S^1),Y]\cong[F_{n-1}(\S^1)\vee\Sigma \mathcal P_{n-1}(\S^1),Y]$,
the identity map $\mbox{id}_Y$ is sent to the corresponding co-action
$$\alpha_{n-1} : \Sigma \mathcal P_{n-1}(\S^1)\to F_{n-1}(\S^1)\vee \Sigma \mathcal P_{n-1}(\S^1).$$
Furthermore, the natural bijection $[\Sigma \mathcal P_{n-1}(\S^1),Y]\times[F_{n-1}(\S^1),Y]\cong
[F_n(\S^1),Y]$ for any pointed space $Y$ yields a homotopy equivalence
$$F_n(\S^1)\simeq F_{n-1}(\S^1)\vee\Sigma \mathcal P_{n-1}(\S^1).$$
\par By means of \cite[Theorem 2.2]{ggw4}, the suspension co-$H$-structure
$$\hat{\mu}_n : F_n(\S^1)\longrightarrow F_n(\S^1)\vee F_n(\S^1)$$
is described inductively and determined by
$$\hat{\mu}^1_n : F_{n-1}(\S^1)\stackrel{\hat{\mu}_{n-1}}{\to}F_{n-1}(\S^1)\vee F_{n-1}(\S^1)\hookrightarrow F_n(\S^1)\vee F_n(\S^1)$$
and a map
$$\hat{\mu}_n^2: \Sigma \mathcal P_{n-1}(\S^1)\to F_n(\S^1)\vee F_n(\S^1)$$
defined via the co-action $\alpha_{n-1} : \Sigma \mathcal P_{n-1}(\S^1)\to F_{n-1}(\S^1)\vee \Sigma \mathcal P_{n-1}(\S^1)$.

\par Similarly to $F_{n+1}(\S^1)$, the suspension co-$H$-structure
$$\overline{\mu}_n: \Sigma J_n(\S^1)\to \Sigma J_n(\S^1) \vee \Sigma J_n(\S^1)$$
is also described inductively.

Because $\Sigma J_1(\S^1)=F_2(\S^1)$, the co-$H$-structure
$\overline{\mu}_1=\hat{\mu}_2 : \Sigma J_1(\S^1)\to \Sigma J_1(\S^1)\vee \Sigma J_1(\S^1)$.
Given the co-$H$-structure $\overline \mu_n : \Sigma J_n(\S^1)
\to \Sigma J_n(\S^1) \vee \Sigma J_n(\S^1)$
write $\Sigma J_{n+1}(\S^1) \simeq \Sigma J_n(\S^1) \vee \S^{n+2}$ and
$F_{n+2}(\S^1) \simeq F_{n+1}(\S^1)
\vee \Sigma \mathcal P_{n+1}(\S^1)$. Then, one can easily verify that the composite maps
$$\overline{\mu}_{n+1}^1 :
\Sigma J_n(\S^1) \stackrel{\overline
{\mu}_n}{\longrightarrow} \Sigma J_n(\S^1) \vee
\Sigma J_n(\S^1) \hookrightarrow \Sigma J_{n+1}(\S^1) \vee
\Sigma J_{n+1}(\S^1)$$
and $$\overline{\mu}_{n+1}^2 :
\S^{n+2} \hookrightarrow \Sigma \mathcal P_{n+1}(\S^1)
\stackrel{\hat{\mu}_{n+2}^2}{\longrightarrow}F_{n+2}(\S^1)\vee F_{n+2}(\S^1)\stackrel{p_{n+1}\vee
p_{n+1}}{\longrightarrow} \Sigma J_{n+1}(\S^1) \vee \Sigma J_{n+1}(\S^1)$$
lead to the suspension co-$H$-structure $$\overline \mu_{n+1}=\overline{\mu}_{n+1}^1\vee\overline{\mu}_{n+1}^2 : \Sigma J_{n+1}(\S^1)
\longrightarrow \Sigma J_{n+1}(\S^1) \vee \Sigma J_{n+1}(\S^1)$$
on the space $\Sigma J_{n+1}(\S^1)\simeq\Sigma J_n(\S^1)\vee \S^{n+2}$.

\bigskip

\par To analyze the suspension co-$H$-structure on $\Sigma J_n(\mathbb S^1)$, we recall the recent work \cite{AL2} on co-$H$-structures
on a wedge of spheres $\S=\bigvee_{i=1}^t\mathbb{S}^{n_i}$. Write
$k_i :\mathbb{S}^{n_i}\hookrightarrow \S$ for the inclusion maps with $i=1,\ldots,t$.
Further, set $\iota_j :\S\hookrightarrow \S\vee \S$ for the inclusion, and
$p_j : \S\vee \S\to \S$ for the projection maps with $j=1,2$. In \cite{AL2}, Arkowitz and Lee have proved the following result.
\begin{proposition}\mbox{$($\cite[Lemma 3.3]{AL2}$)$}\label{n-spheres}{\em
Let $\varphi:\S\to \S\vee \S$ be a co-action. Then, $\varphi$ is a co-$H$-structure on $\S$
if and only if $\varphi k_i=\iota_1k_i+\iota_2k_i+P_i$, where
$P_i : \mathbb{S}^{n_i}\to \S\vee \S$ has the property $p_1P_i=0=p_2P_i$ for $i=1,\ldots,t$.}
\end{proposition}

Adapting Proposition \ref{n-spheres} to our setting, we prove the first of our main results.
\begin{theorem} \label{DMP} The suspension co-$H$-structure $$\overline{\mu}_n : \Sigma J_n(\mathbb S^1)
\to \Sigma J_n(\mathbb S^1) \vee \Sigma J_n(\mathbb S^1)$$ for $n\ge 1$ is given by $\overline{\mu}_n k_i\simeq\iota_1k_i+\iota_2k_i+P_i$, where the perturbation $P_i\simeq\sum_{l=0}^i\phi(i-l,i-1)P_{l,i}$  with the Fox function $\phi$ and
$$P_{l,i}: \mathbb S^{i+1}\to \mathbb S^{i-l+1}\vee \mathbb S^{l+1}\stackrel{k_{i-l}\vee k_l}{\hookrightarrow}\Sigma J_n(\mathbb S^1)\vee \Sigma J_n(\mathbb S^1)$$
determined by the Whitehead product map $\mathbb S^{i+1}\to \mathbb S^{i-l+1}\vee \mathbb S^{l+1}$ for $i=0,\ldots,n$ and $l=0,\ldots,i$.
\end{theorem}
\begin{proof} We proceed inductively on $n\ge 1$. Since the space
$J_n(\mathbb S^1)$ is a $CW$-complex, by the Cellular Approximation Theorem,
$\overline{\mu}_n\simeq \overline{\mu}'_n : \Sigma J_n(\mathbb S^1)
\to \Sigma J_n(\mathbb S^1) \vee \Sigma J_n(\mathbb S^1)$, where $\overline{\mu}'_n$ is cellular.
Hence, from now on, we may assume that $\overline{\mu}_n$ is a cellular map.
\par If $n=1$ then $P_0=0$ and, by definition, $\bar{\mu}_1k_0=\iota_1k_0+\iota_2k_0$.
\par Because $J_n(\mathbb S^1)\subseteq J_{n+1}(\mathbb S^1)$, for dimensional reasons,
we derive that the restriction $\overline{\mu}_{n+1}|_{J_n(\mathbb S^1)}=\overline{\mu}_n$.
Consequently, we must analyze the map $\overline{\mu}_{n+1} k_{n+1}: \mathbb S^{n+2}\to
\Sigma J_{n+1}(\mathbb S^1)\vee \Sigma J_{n+1}(\mathbb S^1)$ determined by the composition
$$\overline{\mu}_{n+1}^2 :
\mathbb S^{n+2} \hookrightarrow \Sigma \mathcal P_{n+1}(\mathbb S^1)
\stackrel{\hat{\mu}_{n+2}^2}{\longrightarrow} F_{n+2}(\S^1)\vee F_{n+1}(\S^1)\stackrel{p_{n+1}\vee
p_{n+1}}{\longrightarrow} \Sigma J_{n+1}(\mathbb S^1)\vee \Sigma J_{n+1}(\mathbb S^1).
$$
But, in view of the decompositions (\ref{eq1}) and (\ref{eq2}), we derive that
$$p_{n+1}|_{\Sigma \mathcal P_{n+1}(\mathbb S^1)} : \Sigma \mathcal P_{n+1}(\mathbb S^1)\simeq \bigvee_{k=1}^{n+1}\binom{n}{k-1}\mathbb S^{k+1}
\to \bigvee _{k=1}^{n+1}\mathbb S^{k+1}$$
restricts to the inclusion map $\S^{k+1}\hookrightarrow \bigvee _{k=1}^{n+1}\mathbb S^{k+1}$ for $k=1,\ldots,n+1$.
\par Because $\hat{\mu}_{n+2}^2$ is defined via the co-action $\alpha_{n+1}: \Sigma \mathcal P_{n+1}(\mathbb S^1) \to F_{n+1}(\mathbb S^1) \vee \Sigma \mathcal P_{n+1}(\mathbb S^1)$,
granting the monomorphism $[\Sigma J_n(\S^1),Y] \to [F_{n+1}(\S^1),Y]=\tau_{n+1}(Y)$ analysed in \cite[Theorem 2.2]{ggw5}, we have
to include all possible ways of contributing the same Whitehead product and thus the sum of all those $(-1)^{w+(|{\bf a}|-1)}$
given by Proposition  \ref{gen-Fox-W-product}.
This leads to $\overline{\mu}_{n+1}k_{n+1}-\iota_1k_{n+1}-\iota_2k_{n+1}\simeq P_{n+1}\simeq\sum_{i=0}^{n+1}\phi(n+1-i,n) P_{n+1,i}$
with $P_{n+1,i}: \mathbb S^{n+2}\to \mathbb S^{n+2-i}\vee \mathbb S^{i+1}\hookrightarrow J_{n+1}(\mathbb S^1)\vee J_{n+1}(\mathbb S^1)$
determined by the Whitehead product map $\mathbb S^{n+2}\to \mathbb S^{n+2-i}\vee \mathbb S^{i+1}$
and the proof is complete.
\end{proof}

When $Y$ is a $W$-space, that is, all Whitehead products vanish, it has
been shown in \cite[Corollary 3.6]{ggw5} that $[J(\S^1),\Omega Y]$ is isomorphic to the direct product of $\prod_{i= 2}^\infty\pi_i(Y)$
so that $[J(\S^1),\Omega Y]$ is abelian, in particular. As an easy consequence of Theorem \ref{DMP}, we have
the following corollary.
\begin{corollary}\label{abelian-cohen-groups}{\em
Let $Y$ be a path connected pointed space such that the Whitehead products
$[f,g]: \mathbb S^{k+l+1}\to Y$ for $f: \mathbb S^{k+1}\to Y$
and $g : \mathbb S^{l+1}\to Y$ vanish with $k,l$ even. Then,
the groups $[J_n(\mathbb S^1),\Omega Y]$ for $n\ge 1$ and $[J(\mathbb S^1),\Omega Y]$
are abelian provided $Y$ is a path connected pointed
space with $\pi_i(Y)=0$ for $i$ odd.}
\end{corollary}

Next, we make use of the function $\phi$ of Section 2 to determine whether $\alpha \in \pi_{n+1}(Y)$ and $\beta \in \pi_{m+1}(Y)$ commute in the group $[J(\mathbb S^1),\Omega Y]$.

\begin{theorem}\label{abelian}
Let $\alpha\in \pi_{n+1}(Y), \beta\in \pi_{m+1}(Y)$. Suppose the Whitehead product $[\alpha,\beta]\ne 0$ and has order $k$. Then:
\begin{itemize}
\item if $k=\infty$ then the product $\alpha  \#\beta \in [J(\mathbb S^1),\Omega Y]$ coincides with $\beta \#\alpha$ iff both $n$ and $m$ are odd; \\
\item if $k<\infty$ then the product $\alpha  \#\beta \in [J(\mathbb S^1),\Omega Y]$ coincides with $\beta \#\alpha$ iff both $n$ and $m$ are odd or
\[
k~\text{divides} \begin{cases}
             \binom{\frac{n+m-1}{2}}{\frac{m}{2}} & \text{if $n$ is odd and $m$ is even;} \\
             \binom{\frac{n+m-1}{2}}{\frac{n}{2}} & \text{if $n$ is even and $m$ is odd;} \\
             \binom{\frac{n+m}{2}}{\frac{n}{2}} & \text{if $n$ and $m$ is both even.}
            \end{cases}
            \]
\end{itemize}
\end{theorem}
\begin{proof}
First of all, if there are non-trivial perturbations $P_{l,i}$ determined by the Whitehead product $[\alpha,\beta]$ then $l=m$ and $i-l=n$. In other words, $P_{l,i}=P_{m,n+m}$. It follows that $P_{n+m}=\phi(n,n+m-1)P_{m,n+m}$. Now, we have
$$
\alpha  \# \beta =\alpha + \beta + \phi(n,n+m-1)[\alpha,\beta] \qquad \text{and} \qquad \beta \# \alpha =\alpha + \beta + \phi(m,n+m-1)[\beta,\alpha].
$$
Since $[\beta,\alpha]=(-1)^{(n+1)(m+1)}[\alpha,\beta]$, it suffices to compare $\phi(n,n+m-1)$ with $\phi(m,n+m-1)(-1)^{(n+1)(m+1)}$.

Let $\Delta=|\phi(n,n+m-1)-\phi(m,n+m-1)(-1)^{(n+1)(m+1)}|$. Depending on the parity of $n$ and $m$, we have the following table:

\begin{center}
\begin{tabular}{|c|c|c|c|c|l|}
  \hline
$n$ & $m$ & $\phi(n,n+m-1)$ & $\phi(m,n+m-1)(-1)^{(n+1)(m+1)}$ & $\Delta$ \\ \hline
{odd}    &  {odd}     & $\binom{\frac{n+m-2}{2}}{\frac{n-1}{2}}$ & $\binom{\frac{n+m-2}{2}}{\frac{m-1}{2}}$ & $0$\\ \hline
{odd}    &  {even}    & $0$ & $-\binom{\frac{n+m-1}{2}}{\frac{m}{2}}$ & $\binom{\frac{n+m-1}{2}}{\frac{m}{2}}$\\ \hline
{even}   &  {odd}    & $-\binom{\frac{n+m-1}{2}}{\frac{n}{2}}$ & $0$ & $\binom{\frac{n+m-1}{2}}{\frac{n}{2}}$\\ \hline
{even}   &  {even}   & $-\binom{\frac{n+m-2}{2}}{\frac{n}{2}}$ & $\binom{\frac{n+m-2}{2}}{\frac{m}{2}}$ & $\binom{\frac{n+m}{2}}{\frac{n}{2}}$\\ \hline
\end{tabular}
\end{center}
\begin{equation}\label{table}
\text{Table for }\Delta
\end{equation}

Because of $\binom{a}{b}=\binom{a}{a-b}$, the equality $\Delta=0$ holds exactly when $n$ and $m$ are both odd.
For the case when both $n$ and $m$ are even, we note
that $\binom{\frac{n+m-2}{2}}{\frac{m}{2}}=\binom{\frac{n+m-2}{2}}{\frac{n}{2}-1}$ and the Pascal triangle
asserts that
$$
\Delta = \binom{\frac{n+m-2}{2}}{\frac{n}{2}-1} + \binom{\frac{n+m-2}{2}}{\frac{n}{2}} = \binom{\frac{n+m}{2}}{\frac{n}{2}}.
$$
This completes the proof.
\end{proof}
 Now, using the fact that the group structure of $[J_n(\S^1),\Omega Y]$ is induced by the group structure of
the torus homotopy group $\tau_{n+1}(Y)$ which in turn is determined completely by the homotopy
groups $\{\pi_i(Y)\}_{1\le i\le n+1}$ and their Whitehead products, we give explicitly the group
structure of $[J_n(\S^1),\Omega Y]$ following Theorem \ref{DMP} and using the function $\phi$.

Since $[J_n(\S^1),\Omega Y]$ is in one-to-one correspondence with $\displaystyle{\prod_{i=2}^{n+1}\pi_i(Y)}$ as sets, we denote the group multiplication determined
by the co-$H$-structure on $\Sigma J_n(\mathbb{S}^1)$ by $\#_n$, i.e.,
$$
[J_n(\S^1),\Omega Y]\cong \left(\prod_{i=2}^{n+1} \pi_i(Y), \#_n\right).
$$

Now, we describe $\#_n$ inductively as follows. First write
$$\prod_{i=2}^{n+1} \pi_i(Y)=\left(\prod_{i=2}^{n} \pi_i(Y)\right)\times \pi_{n+1}(Y).$$ For any $\displaystyle{\alpha \in \prod_{i=2}^{n+1} \pi_i(Y)}$, write $\alpha=(\alpha_1,\alpha_2)$, where $\displaystyle{\alpha_1\in \prod_{i=2}^{n} \pi_i(Y)}$ and $\alpha_2\in \pi_{n+1}(Y)$. Moreover, $(\alpha)_k$ denotes the coordinate of $\alpha$ in $\pi_k(Y)$. For $\displaystyle{(\alpha_1,\alpha_2), (\beta_1,\beta_2)\in  \left(\prod_{i=2}^{n} \pi_i(Y)\right)\times \pi_{n+1}(Y)}$,
we have
\begin{equation}\label{J_n-product}
(\alpha_1,\alpha_2)\#_n(\beta_1,\beta_2):=\left(\alpha_1\#_{n-1}\beta_1, \alpha_2+\beta_2+\sum_{k+j=n+2}\phi(k-1,k+j-3)[(\alpha_1)_k,(\beta_1)_j]\right).
\end{equation}
Recall that the natural inclusion $j_n :J_{n-1}(\S^1) \hookrightarrow J_n(\S^1)$ induces a surjective homomorphism $j_n^\ast: [J_n(\S^1),\Omega Y] \to [J_{n-1}(\S^1),\Omega Y]$, where $\mbox{Ker}\, j_n^\ast\cong \pi_{n+1}(Y)$ is central. In other words,
$$
j_n^\ast((\alpha)_1, (\alpha)_2,\ldots,(\alpha)_n, (\alpha)_{n+1})=((\alpha)_1, (\alpha)_2,\ldots,(\alpha)_n)
$$
for $((\alpha)_1, (\alpha)_2,\ldots,(\alpha)_n, (\alpha)_{n+1})\in[\Sigma J_n(\mathbb{S}^1),Y]$.

\section{Computations}

In this section, we revisit, simplify, and generalize the examples from Section 1 using the group structure of $[J_n(\S^1),\Omega Y]$ in the previous
section together with the function $\phi$. Furthermore, we determine whether $[J_k(\mathbb S^1),\Omega \mathbb S^{2n}]$ is abelian for certain values of $k$.

\begin{example}\label{ex1-revisit}
The multiplication in $[J_2(\mathbb S^1),\Omega \mathbb S^2]$ as in Example \ref{ex1} is given by the following rule:
$$(\alpha_1,\alpha_2) \#_2 (\beta_1,\beta_2)=(\alpha_1+\beta_1,\alpha_2+\beta_2+2\alpha_1\beta_1),
$$
where $\alpha_1,\beta_1\in \pi_2(\S^2)\cong \mathbb Z$ and $\alpha_2,\beta_2 \in \pi_3(\S^2)
\cong \mathbb Z$.

The perturbation $P: \S^3\to \Sigma J_2(\S^1)\vee \Sigma J_2(\S^1)$ in this case is
determined by the basic Whitehead product $[i_1,i_2]$, where
$i_j :\S^2\to\S^2\vee\S^2$ is the corresponding inclusion for $j=1,2$
which yields in the group $[J_2(\mathbb S^1),\Omega \mathbb S^2]$ the relation $[\iota_2,\iota_2]=2\eta_2$
for the generators $\iota_2\in\pi_2(\S^2)$ and $\eta_2\in\pi_3(\S^2)$
given by the Hopf map. Further, the isomorphism
$\mathbb Z\oplus \mathbb Z\cong [J_2(\S^1),\Omega \S^2]$ is given
by $(m,n) \mapsto(m,m(m-1)+n)$ for $(m,n)\in\mathbb Z\oplus \mathbb Z$.
\end{example}

More generally, we can use Theorem \ref{DMP} to analyze the co-$H$-structure of $\Sigma J_2(\mathbb S^1)$.
The perturbation $P: \mathbb S^{3}\to \Sigma J_2(\mathbb S^1)\vee \Sigma J_2(\mathbb S^1)$ is given by the
Whitehead product $\mathbb S^{3}\to \mathbb S^2\vee \mathbb S^2$.
Therefore, the multiplication on the set $[J_2(\mathbb S^1), \Omega Y]=\pi_2(Y) \times \pi_3(Y)$
is given by
$$(\alpha_1,\alpha_2)\#_2(\beta_1,\beta_2)=(\alpha_1+\beta_1,\alpha_2+\beta_2+[\alpha_1,\beta_1]),$$
where $\alpha_1,\beta_1\in \pi_2(Y)$, $\alpha_2,\beta_2\in \pi_3(Y)$ and
$[\alpha_1,\beta_1]$ denotes the Whitehead product.
Note that $[\beta_1,\alpha_1]=(-1)^4[\alpha_1,\beta_1]=[\alpha_1,\beta_1]$. Consequently, we generalize Example \ref{ex1} by the following proposition.
\begin{proposition}\label{J_2-abelian}{\em
The group $[J_2(\mathbb S^1),\Omega Y]$ is an abelian group.}
\end{proposition}

\begin{remark}
Since $[J_2(\mathbb S^1),\Omega Y]$ is abelian, using induction on the central extension $0\to \pi_{n+1}(Y) \to [J_n(\S^1),\Omega Y]\to [J_{n-1}(\S^1),\Omega Y]\to 1$ shows that the group $[J_n(\S^1),\Omega Y]$ is nilpotent with nilpotency class $\le n-1$.
\end{remark}

\begin{proposition} {\em If $Y$ is a $(2n-1)$-connected space then there is
an isomorphism of groups $[J_{4n-3}(\S^1),\Omega Y]\cong\pi_{2n}(Y)\oplus\pi_{2n+1}(Y)\oplus\cdots\oplus\pi_{4n-2}(Y)$,
the group $[J_{4n-2}(\S^1),\Omega Y]$ is abelian and the short exact sequence
$$0\to\pi_{4n-1}(Y)\longrightarrow[J_{4n-2}(\S^1),\Omega Y]\longrightarrow[J_{4n-3}(\S^1),\Omega Y]\to 0$$
splits provided $\pi_{2n}(Y)$ is free abelian.
In particular, this holds for $Y=\S^{2n}$.}
\end{proposition}
\begin{proof} For dimensional reasons and by the connectivity of $Y$, there are no non-trivial Whitehead products obtained from the elements in $[J_{4n-3}(\S^1),\Omega Y]$, which in turn is isomorphic to $\pi_{2n}(Y)\oplus\pi_{2n+1}(Y)\oplus\cdots\oplus\pi_{4n-2}(Y)$. The only possibly non-trivial Whitehead products in $[J_{4n-2}(\S^1),\Omega Y]$ come from elements $\alpha, \beta \in \pi_{2n}(Y)$. It follows from Table \eqref{table} that $\Delta =0$, that is, $[\alpha, \beta]=[\beta,\alpha]$. This implies that $[J_{4n-2}(\S^1),\Omega Y]$ is abelian. Finally, when $\pi_{2n}(Y)$ is free abelian, one can find a section $\pi_{2n}(Y) \to [J_{4n-2}(\S^1),\Omega Y]$ so that $\alpha \#_{4n-3} \beta \mapsto (\alpha \#_{4n-3} \beta, [\alpha,\beta])$.
Hence, this gives rise to a section $\pi_{2n}(Y)\oplus\pi_{2n+1}(Y)\oplus\cdots\oplus\pi_{4n-2}(Y) \to[J_{4n-2}(\S^1),\Omega Y]$.
\end{proof}

\begin{remark}
The assumption that $\pi_{2n}(Y)$ being free abelian above is only sufficient for the splitting of the short exact sequence as we illustrate in the following examples.
\end{remark}

\par It is well-known that the $(n-2)$-suspension $M^n=\Sigma^{n-2}\mathbb R P^2$ of the projective plane $\mathbb R P^2$
is the Moore space of type $(\Z_2,n-1)$ for $n\ge 3$ so that $M^n$ is $(n-2)$-connected.
Given the inclusion $i_2 : \S^1\hookrightarrow \mathbb R P^2$ and the collapsing map $p_2 : \mathbb R P^2\to \mathbb R P^2/\S^1=\S^2$,
we write $i_n =\Sigma^{n-2}i_2 : \S^{n-1}\to M^n$ and $p_n = \Sigma^{n-2}p_2 : M^n\to \S^n$ for the  $(n-2)$-suspension maps with $n\ge 2$, respectively.

\begin{example}\label{JM-ex1}
Take $Y=M^3$. In view of \cite[Proposition 3.6]{G-M}, $\pi_2(M^3)= \mathbb Z_2\langle i_3\rangle$ and $\pi_3(M^3)=\mathbb Z_4\langle i_3\circ \eta_2\rangle$.
Then, the following short exact sequence
$$0\to\pi_{3}(M^3)\longrightarrow[J_{2}(\S^1),\Omega M^3]\longrightarrow[J_{1}(\S^1),\Omega M^3]\to 0$$
becomes
\begin{equation}\label{Mukai_ex1}
0\to \mathbb Z_4 \to [J_{2}(\S^1),\Omega M^3] \to \mathbb Z_2 \to 0.
\end{equation}
By Proposition \ref{J_2-abelian}, $[J_{2}(\S^1),\Omega M^3]$ is abelian and thus is isomorphic to either $\mathbb Z_8$ or $\mathbb Z_4 \oplus \mathbb Z_2$.
It was pointed out to us by J.\ Mukai that the Whitehead product $[i_3,i_3]\ne 0$ and its order is $2$. It is straighforward to see, based upon the group structure of $[J_{2}(\S^1),\Omega M^3]$, that there are no elements of order $8$ in $[J_{2}(\S^1),\Omega M^3]$.
Hence, the short exact sequence \eqref{Mukai_ex1} splits.
\end{example}

Next, we provide a similar example \footnote{The authors are grateful to Juno Mukai for providing the example below.} as above, where the sequence does not split.

\begin{example}\label{JM-ex2}
The  short exact sequence
$$0\to \pi_{11}(M^7)\longrightarrow[J_{10}(\S^1),\Omega M^7]\longrightarrow[J_9(\S^1),\Omega M^7]\cong\pi_6(M^7)\oplus\cdots\oplus\pi_{10}(M^7)\to 0$$
is central, does not split, and the group $[J_{10}(\S^1),\Omega M^7]$ is abelian.

To see this, we recall that $\pi_{6}(M^7)=\Z_2\langle i_7\rangle$, and in view of \cite{MS,wu}, it follows
that $\pi_{11}(M^7)\cong \Z_2\langle [i_7,i_7]\rangle$. Consider the element $\alpha =(0,0,0,0,i_7,0,0,0,0)$. Then
by \eqref{J_n-product}, we get
$$(\alpha,0) \#_{10} (\alpha,0) = (2\alpha, [i_7,i_7]) =(0,[i_7,i_7]) \quad \text{since $\alpha$ has order $2$.}$$
It follows that
$$(\alpha,0) \#_{10} (\alpha,0) \#_{10} (\alpha,0) \#_{10} (\alpha,0) =(0,0)  \quad \text{since the Whitehead product $[i_7,i_7]$ has order 2.}$$

This shows that the group $[J_{10}(\S^1),\Omega M^7]$ contains a copy of the group $\Z_4$ determined
by the element $\alpha$.
Hence, $[J_{10}(\S^1),\Omega M^7]$ is abelian but it is not isomorphic to the direct sum $\pi_6(M^7)\oplus\cdots\oplus\pi_{11}(M^7)$. Consequently,
the short exact sequence
$$0\longrightarrow \pi_{11}(M^7)\longrightarrow[J_{10}(\S^1),\Omega M^7]\longrightarrow[J_9(\S^1),\Omega M^7]\cong\pi_6(M^7)\oplus\cdots\oplus\pi_{10}(M^7)\longrightarrow 0$$
does not split.
\end{example}

In Example \ref{ex3}, the group $[J_{k}(\S^1),\Omega \S^{4}]$ is non-abelian for $k\ge 7$. We now generalize this example in the following proposition.

\begin{proposition} \label{1-2_stems}{\em
\mbox{\em (1)} The group $[J_{4n-1}(\S^1),\Omega \S^{2n}]$ is non-abelian if and only if
$n$ is a power of $2$. In particular, when $n$ is a power of $2$, $[J_k(\S^1),\Omega \S^{2n}]$ is non-abelian for $k\ge 4n-1$.
When $n$ is not a power of $2$, $[J_k(\S^1),\Omega \S^{2n}]$ is abelian for $k\leq 4n$.
\par \mbox{\em (2)} The group $[J_k(\S^1),\Omega M^{2n}]$ is non-abelian for $k\ge 4n-1$ provided $n$ is a power of $2$.}
\end{proposition}
\begin{proof}
(1): Given a prime $p$, consider the base $p$ expansions of the integers $m$ and $n$, where we
assume $0\le m\le n$:
\noindent
$n=a_0+pa_1+\cdots+p^ka_k,$ $0\le a_i<p$ and $m=b_0+pb_1+\cdots+p^kb_k$, $0\le b_i< p$.
Then, by Lukas' Theorem (see e.g., \cite{fine}),
$${n\choose m}\equiv \prod^k_{i=0}{a_i\choose b_i}(\bmod\; p).$$
Hence, ${n\choose m}\not\equiv 0\;(\bmod\; p)$ if and only if $b_i\le a_i$ for all $0\le i\le k$.
\par In particular, for $0\le m\le n$ and $p=2$ with expansions
$n=a_0+2a_1+\cdots+2^ka_k,$ $0\le a_i<2$ and $m=b_0+2b_1+\cdots+2^kb_k$, $0\le b_i< 2$,
the number ${n\choose m}$ is odd if and only if $b_i\le a_i$ for all $0\le i\le k$.
Thus, following Table (\ref{table}), $\Delta={2n-1\choose n}$ is odd if and only if $n=2^l$ for some $l\ge 0$.
Since $[\eta_{2n},\iota_{2n}]\ne 0$ (see \cite[p.\ 404]{GM}) and $\Delta\ne 0$, we have $(0,\ldots,0,\iota_{2n},0\ldots,0)\#_k(0,\ldots,0,\eta_{2n},0\ldots,0)\ne(0,\eta_{2n},0\ldots,0)\#_k(\iota_{2n},0\ldots,0)$
and consequently $[J_{4n-1}(\S^1),\Omega \S^{2n}]$ is non-abelian if, and only if,
$n$ is a power of $2$. Hence, $[J_k(\S^1),\Omega\S^{2n}]$ is non-abelian for $k\ge 4n-1$ when $n=2^l$ for some $l>0$.

When $n$ is not a power of $2$, we only need to consider the Whitehead product $[\iota_{2n},\eta_{2n}^2]$ because other Whitehead products from lower
dimensions lie in the abelian group $[J_{4n-1}(\S^1),\Omega \S^{2n}]$. Since both $2n-1$ and $(2n+2)-1$ are both odd, it follows from Table (\ref{table})
that $\Delta=0$, i.e., $\iota_{2n}\#_{4n}\eta_{2n}^2=\eta_{2n}^2\#_{4n}\iota_{2n}$. Thus, we conclude that $[J_{4n}(\S^1),\Omega \S^{2n}]$ is abelian and hence so is  $[J_{k}(\S^1),\Omega \S^{2n}]$ for any $k\le 4n$.

(2): Let $\tilde{\eta}_2\in \pi_4(M^3)$ be a lift of $\eta_3\in\pi_4(\S^3)$ satisfying $2\tilde{\eta}_2=i_n \eta^2_2$,
$p_3\tilde{\eta}_2=\eta_3$ and set $\tilde{\eta}_n=\Sigma^{n-2}\tilde{\eta}_2$ for $n\ge 2$.
Then $\pi_{2n+1}(M^{2n})= \Z_4\langle \tilde{\eta}_{2n-1}\rangle$ and, by \cite[Lemma 3.8]{G-M}, it holds
$[i_{2n}\eta_{2n-1},\tilde{\eta}_{2n-1}]\ne 0$ for $n\ge 2$. Then, we can deduce as in (1)
that $[J_k(\S^1),\Omega M^{2n}]$ is non-abelian for $k\ge 4n-1$ and $n=2^l$  for some $l>0$.
\end{proof}

Based upon the results in Proposition \ref{1-2_stems}, the next question is whether $[J_{4n+1}(\S^1),\Omega\S^{2n}]$ is abelian when $n$ is not a power of $2$. Proposition \ref{1-2_stems} (1) depends on certain divisibility properties of certain types of binomial coefficients. In the next result, we answer
this question by exploring further such divisibility results concerning the Catalan numbers and thereby strengthen Proposition \ref{1-2_stems}.
Let $T^*(01)$ denote the set of natural numbers $n$ with $(n)_3=(n_i)$ and $n_i\in \{0,1\}$ for $i\ge 1$, where $(n)_3$ is the base $3$ expansion of $n$. Further, denote by $T^*(01)-1$ the set $\{n-1|n\in T^*(01)\}$.
\par Following Table (\ref{table}), we get $\Delta=\binom{2n}{n-1}$. Since $\binom{2n}{n-1}+\binom{2n}{n}=\binom{2n+1}{n}=\frac{2n+1}{n+1}\binom{2n}{n}$,
it follows that
\begin{equation}\label{catalan}
\Delta=\binom{2n}{n-1}=\binom{2n}{n}\left(\frac{2n+1}{n+1}-1\right)=\frac{n}{n+1}\binom{2n}{n}=nC_n,
\end{equation}
where $C_n$ is the $n$-th Catalan number.
\par Write $\nu_{2n}$ for a generator of $\pi_{2n+3}(\S^{2n})$ with $n\ge 2$ and consider the Whitehead product $[\iota_{2n},\nu_{2n}]$.
Suppose $n$ is not a power of $2$. According to \cite[p.\ 405]{GM}, the order of $[\iota_{2n},\nu_{2n}]$ is $12$ if $n$ is odd or of order $24$ if $n$ is even.

\begin{proposition}\label{3_stem}{\em
Suppose $n\ne 2^{\ell}$ for any $\ell>0$. Then $[J_{4n+1}(\S^1),\Omega \S^{2n}]$ is abelian if and only if,
\begin{itemize}
\item[(1)] \mbox{\em (i)} $n\ne 2^a-1$ and $n\ne 2^a+2^b-1$ for some $b>a\ge 0$;
\noindent
and
\noindent
\mbox{\em (ii)} $n\equiv 0\, (\bmod\, 3)$ or $n\notin T^*(01)-1$ when $n$ is odd;
\item[(2)] $n\equiv 0\, (\bmod \,3)$ or $n\notin T^*(01)-1$ when $n$ is even.
\end{itemize}}
\end{proposition}
\begin{proof}
When $n$ is odd, $\Delta=nC_n$ is divisible by $12$ if and only if, $C_n$ is divisible by $4$ and either $n$ or $C_n$ is divisible by $3$. Similarly, when $n$ is even,  $\Delta=nC_n$ is divisible by $24$ if and only if, $C_n$ is divisible by $4$ (since $n$ is even) and either $n$ or $C_n$ is divisible by $3$.
The result follows from \cite[Theorem 5.2]{ds} for the divisibility of $C_n$ by $3$ and from \cite[Theorem 2.3]{ely} for the divisibility of $C_n$ by $4$.
\end{proof}

\begin{remark}
Let $n=29$. Then, $C_{29}$ is divisible by $4$ but $29$ is not divisible by $3$ and $C_{29}$ is not divisible by $3$ since $29\in T^*(01)-1$. Thus, in case
(1) in Proposition \ref{3_stem}, (i) is satisfied but (ii) is not.

Let $n=34$. Then, $C_{34}$ is divisible by $4$ and by $3$ but $34$ is not divisible by $3$.
\end{remark}

The following result generalizes Example \ref{ex2}.
\begin{proposition}{\em
The group $[J_{4n-1}(\S^1),\Omega(\S^{2n}\vee\S^{2n+1})]$ is non-abelian.}
\end{proposition}
\begin{proof}
Consider the inclusion maps $i_1 : \S^{2n}\hookrightarrow\S^{2n}\vee\S^{2n+1}$ and $i_2 :\S^{2n+1}\hookrightarrow\S^{2n}\vee\S^{2n+1}$.
Then, Hilton's result \cite{H} asserts that for every positive integer $k$, there is an isomorphism
$$\Theta :\bigoplus_{l=1}^\infty\pi_k(\S^{n_l})\stackrel{\cong}{\longrightarrow} \pi_k(\S^{2n}\vee\S^{2n+1}),$$
where the restriction $\Theta|_{\pi_k(\S^{n_l})}=\omega_{l\ast} :\pi_k(\S^{n_l})\to \pi_k(\S^{2n}\vee\S^{2n+1})$
is determined by the iterated Whitehead product of the maps $i_1$ and $i_2$.
In particular, the Whitehead product $[i_1,i_2] : \S^{4n}\to  \S^{2n}\vee\S^{2n+1}$ is non-trivial.
Furthermore, $$(0,\ldots,0,i_1,0,\ldots)\#_{4n-1}(0,\ldots,0,i_2,0,\ldots)=(0,\ldots,0,i_1,0,\ldots,0,i_2,0,\ldots)$$ and
$$(0,\ldots,0,i_2,0,\ldots)\#_{4n-1}(0,\ldots,0,i_1,0,\ldots)=(0,\ldots,0,i_1,0,\ldots,0,i_2,0,\ldots,0,[i_1,i_2],0\ldots).$$
The above implies that the group $[J_{4n-1}(\S^1),\Omega(\S^{2n}\vee\S^{2n+1})]$ is non-abelian.
\end{proof}

To close the paper, we derive few simple properties about the torsion elements in $[J(\S^1), \Omega Y]$.
\begin{proposition} {\em Let $\alpha, \beta\in [J(\S^1), \Omega Y]$ be two elements which correspond to
homogeneous sequences with $\alpha\in \pi_m(Y)$ and $\beta\in \pi_n(Y)$.  Then:

\mbox{\em (1)} if $\alpha$ has order $k$ in $\pi_m(Y)$, then $\alpha$, regarded as an element
of $[J(\S^1), \Omega Y]$, has order $k$ or $k^2$;

\mbox{\em (2)} if $\alpha,\beta$ are torsion elements of order $|\alpha|$ and $|\beta|$, respectively such that ${\rm gcd}(|\alpha|,|\beta|)=1$, then $\alpha$ and $\beta$ commute in $[J(\S^1), \Omega Y]$.}
\end{proposition}
\begin{proof} (1): By \cite[Chapter XI, Section 8, Theorem 8.8]{Whi}, all Whitehead products of weight $\geq 3$ of an element of  odd dimension and all Whitehead products of weight $\geq 4$ of an  element of even dimension, vanish. Therefore, using our formula and the result above,
we obtain that $\alpha^k$ (as an element of $[J(\S^1), \Omega Y]$) is a sequence of the form $(0,\ldots,0, k\alpha, \lambda_1[\alpha, \alpha], \lambda_2[\alpha, [\alpha, \alpha]],0,\ldots)=
(0\ldots,0, \lambda_1[\alpha, \alpha], \lambda_2[\alpha, [\alpha, \alpha]],0,\ldots)$. Again, by the result cited above, we obtain that
\begin{equation*}
\begin{aligned}
(0,\ldots,0,\lambda_1[\alpha, \alpha], \lambda_2[\alpha, [\alpha, \alpha]],0,\ldots)^k&=(0,\ldots,0, k\lambda_1[\alpha, \alpha], k\lambda_2[\alpha, [\alpha, \alpha]],0,\ldots) \\
&=(0,\ldots,0, \lambda_1[k\alpha, \alpha], \lambda_2[k\alpha, [\alpha, \alpha]],0,\ldots) \\
&=0
\end{aligned}
\end{equation*} and (1) follows.

(2): It suffices to observe that the Whitehead product $[\alpha, \beta]$ vanishes. Since $|\alpha|[\alpha, \beta]=[|\alpha|\alpha, \beta]=0$,
$|\beta|[\alpha, \beta]=[\alpha, |\beta|\beta]=0$ and ${\rm gcd}(|\alpha|,|\beta|)=1$, it follows that $[\alpha, \beta]=0$  and
the proof is complete.
\end{proof}

\end{document}